\index{}

\documentclass[12pt,reqno]{amsart}
\usepackage{amssymb , amsmath, amsthm}
\usepackage{amsmath}
\usepackage{amsbsy}
\usepackage{amssymb}
\usepackage{color}
\usepackage[dvips]{graphicx}
 \usepackage{subfigure}
\usepackage[active]{srcltx}

\usepackage{amsmath}
\usepackage{amsfonts}
\usepackage{amssymb}
\usepackage{amstext}
\usepackage{amsbsy}
\usepackage{epsf}

\usepackage{color}
\usepackage{xcolor}
\usepackage[colorlinks=true,linkcolor=blue,citecolor=red,filecolor=yellow,
urlcolor=green]{hyperref}

\newtheorem{theorem}{Theorem}[section]

\newtheorem{corollary}{Corollary}[section]
\newtheorem{lemma}{Lemma}[section]
\theoremstyle{definition}
\newtheorem{definition}{Definition}[section]
\newtheorem{example}{Example}[section]
\newtheorem{remark}{Remark}[section]

\newtheorem{claim}{\text{Claim}}

\newcommand{\thmref}[1]{Theorem~\ref{#1}}
\newcommand{\secref}[1]{Section~\ref{#1}}
\newcommand{\lemref}[1]{Lemma~\ref{#1}}

\newcommand{\coref}[1]{Corollary~\ref{#1}}

\providecommand{\Real}{\mathop{\rm Re}\nolimits}
\providecommand{\Imag}{\mathop{\rm Im}\nolimits}

\def\r{\mathbb R}

\def\h{\mathbb H}

\def\Hip{\mathbb H}

\let\hip=\Hip

\def \Cx {{\mathbb C}}

\newcommand{\R}{\mathbb R}

\newcommand{\C}{\mathbb C}

\newcommand{\di}{\mathbb D}
\newcommand{\N}{\mathbb N}
\newcommand{\Ne}{{\mathbb N}^*}

\newcommand{\wt}{\widetilde}
\newcommand{\wh}{\widehat}

\renewcommand{\Re}{{\rm Re}}

\let\h=\hip

\let\re=\R

\def\rmd{\mathop{\rm d\kern -1pt}\nolimits}
\def\rme{\mathop{\rm e\kern -1pt}\nolimits}

\def\bel{ \medskip
 \centerline{$ \ast \hbox to 1.0cm{}\ast \hbox to 1.0cm{}\ast $}
}

\def\goto{\rightarrow}

\def\longerrightarrow{-\kern-5pt\longrightarrow}

\def\star{\lower 1pt\hbox{*}}
\def \nulset {
\raise 1pt\hbox{ \hskip -3pt$\not$\kern -0.2pt \raise
.7pt\hbox{${\scriptstyle\bigcirc}$}}}

\newcommand{\hd}{\mathbb{H}^2}
\newcommand{\sd}{\mathbb{S}^2}
\newcommand{\hi}[1]{\mathbb{H}^#1}

\newcommand{\ch}{\cosh}
\newcommand{\sh}{\sinh}

\newcommand{\pain}{\partial_{\infty}}

\newcommand{\ov}[1]{\overline{#1}}

\let\leq=\leqslant
\let\geq=\geqslant

\begin{document}

\title[total curvature of minimal surfaces]{Concentration of total curvature of
minimal surfaces in $\hi2\times \r$ }

\author[R. Sa
Earp $\ $ and $\ $ E. Toubiana  ]{
 Ricardo Sa Earp and
Eric Toubiana}

 \address{Departamento de Matem\'atica \newline
  Pontif\'\i cia Universidade Cat\'olica do Rio de Janeiro\newline
Rio de Janeiro \newline
22453-900 RJ \newline
 Brazil }
\email{rsaearp@gmail.com}

\address{Institut de Math\'ematiques de Jussieu - Paris Rive Gauche \newline
Universit\'e Paris Diderot - Paris 7 \newline
Equipe G\'eom\'etrie et Dynamique,  UMR 7586 \newline
B\^atiment Sophie Germain \newline
Case 7012 \newline
75205 Paris Cedex 13 \newline
France}
\email{eric.toubiana@imj-prg.fr}

\thanks{
\\
The  first author was partially supported by CNPq
 of Brasil.
}

\date{\today}

\begin{abstract}
 We prove a phenomenon of concentration of total curvature for stable minimal surfaces
 in  the product space $\hip^2 \times \R,$ where $\hip^2$ is the
hyperbolic plane. Under some
geometric
conditions  on the asymptotic boundary  of an oriented  stable
minimal surface
immersed in
$\hip^2 \times \R$, it has infinite  total curvature.
In particular,  we infer that  a minimal graph $M$ in $\hip^2 \times \R$  whose
asymptotic boundary is a graph over an arc of
$\partial_{\infty}\hi2 \times \{0\}$, different from the asymptotic
boundary of $\partial M$,
has infinite  total curvature. Consequently,
if $M$ is a stable minimal surface immersed into $\hip^2 \times \R$ with compact boundary,
such that its  asymptotic boundary
is a
graph over  the whole $\partial_{\infty}\hi2 \times \{0\},$ then it has infinite
total curvature.
We exhibit an example of a minimal graph such that in a domain
whose
asymptotic boundary is a vertical segment the total
curvature is finite, but  the total curvature of the graph  is infinite,
by the theorem  cited before.  We also
present some simple  and peculiar examples of infinite total curvature minimal
surfaces in $\hip^2 \times \R$   and their asymptotic boundaries.

\end{abstract}


\maketitle

\section{Introduction}

In this paper, we prove a phenomenon of concentration of total curvature
for stable minimal surfaces in  the product space $\hip^2 \times
\R$,
where $\hip^2$ is the hyperbolic plane.

We recall that a minimal surface $M$ immersed in $\hip^2\times \R$
 has  {\em finite intrinsic total curvature}, or simply {\em finite total
curvature},  if $\int_M K  \, dA$ is
finite, where $K$ is the  (intrinsic) Gaussian curvature of $M$.

 Our main theorem,
see \thmref{Main theorem},  ensures that under some geometric
conditions  on the asymptotic boundary, an oriented stable minimal
surface  immersed in
$\hip^2 \times \R$  (but not necessarily properly immersed) has
infinite  total curvature.
In particular,  we prove that  a minimal graph $M$ in $\hip^2
\times \R$
whose  asymptotic
boundary is a graph over an arc of $\partial_{\infty}\hi2 \times
\{0\}$,
different from the asymptotic boundary of $\partial M$,
has infinite  total curvature.

 Consequently, we infer the following corollary  of \thmref{Main theorem}:
Let $M$ be an oriented  stable minimal surface immersed into $\hip^2 \times \R$
with
compact boundary (e. g. a minimal graph with compact boundary),
such that its  asymptotic boundary
is a
graph over  the whole $\partial_{\infty}\hi2 \times \{0\}$. Then $M$ has
infinite
total curvature (\coref{C.arc}).

Furthermore, we deduce that  if $p_\infty\in  \pain M $,
 we have $\vert n_3(p)\vert \to 1$ if
$p \to p_\infty$, $p\in M$, where $n_3$ is the third coordinate of the
unit normal field  (\coref{C. end}).

\vskip1.5mm
   In  \secref{remarkable}  we exhibit  an example   of
a   minimal graph such that in a domain
whose asymptotic boundary is a vertical segment the total curvature is finite,
but the total curvature of the graph  is
infinite by    \thmref{Main theorem}.

\vskip1.5mm

We also
present some simple  and peculiar examples of infinite total curvature minimal
surfaces in $\hip^2 \times \R$   and their asymptotic boundaries.

\vskip2mm

 We would like to point out that in our previous work,  we proved
a  geometric property of an  oriented stable minimal
 annuli end $M$  with compact boundary, properly immersed into $\Hip^2\times
\R$,
see \cite{SE-T3}. We    introduced the
notion of asymptotic boundary of $M$  that is
suitable for our studies, as well. We call the set of points of the
asymptotic boundary of $M$ with finite  (vertical)  height,
the {\em finite asymptotic boundary} (see Definition \ref{D.asymptotic}
below). The main result in \cite{SE-T3} ensures
that if  the end  $M$ converges to a vertical plane
and
the  finite  asymptotic boundary  of $M$ is contained in two vertical
lines,
then $M$ has finite total curvature. If
the end of $M$ is embedded we showed that it is a horizontal
graph with respect to a
horizontal geodesic  $\gamma$, or simply a horizontal graph.
For a definition of  a
horizontal graph, see, for instance, \cite{HNST}.
 Loosely speaking in certain sense  the main result in  \cite{SE-T3}
is a counterpart of the main theorem in the present
paper.  Summarizing, we proved  in \cite{SE-T3}
the following geometric behavior, up to a compact part of $M$:
 \begin{itemize}
 \item  Any  equidistant curve of $\gamma$ intersects the end
$M$ at most at one point and it
 intersects it transversally (\cite[ Proposition A.3]{HNST}
and \cite[Step 7]{SE-T3}.

 \vskip1.5mm

 \item  Let $n_3$ be  the third coordinate of the
unit normal field on $M$ with respect to the product metric
on $\hi2 \times \R$.
We have that $n_3 (p) \to 0$ uniformly when $p \to
\pain E$ \cite[Step 3]{SE-T3}.  Consequently, the tangent plane
throughout the end is nowhere horizontal.
\end{itemize}

It will be very interesting to investigate a similar result as in
\cite{SE-T3}, for minimal surfaces immersed into $\hip^2\times \R$ with
nonempty  non compact boundary. In this
direction we set here the following problem: Find geometric conditions on a
minimal surface $S$ immersed into $\hip^2\times \R,$ with  nonempty  non compact
boundary, whose finite asymptotic boundary
is contained in a vertical line to have  finite total curvature.
The  example exhibited in \secref{remarkable}   suggests  this problem.

\vskip1.5mm

We remark that
 a horizontal graph with respect a geodesic $\gamma$ which is
transverse
to the  intersecting equidistant curves  to $\gamma$ is stable.  A
vertical graph (see the
definition, for instance, in \cite{SE-T2}), which is transverse to the
intersecting
vertical geodesics is stable, as well. This follows from the classical criterion
of stability for minimal surfaces:
 Let $M$ be an oriented connected minimal surface immersed into $\hip^2\times
\R.$
If there exists a positive smooth function $u$ on a bounded domain $\Omega$ of
$M$ satisfying $\mathcal L u=0, $
where $\mathcal L$ is the stability operator \cite[Section 2.2]{B-SE}, then
$\Omega$ is stable \cite[Lemma 1.36]{C-M}.

\vskip1.5mm

 We point out now an important property of a complete minimal
surface immersed in $\hip^2\times \R$:
The finite asymptotic boundary of a
complete
minimal surface in $\Hip^2\times \R$ with finite total curvature is constituted
of vertical lines, see  \cite[Theorem 2.1 and
Proposition 2.4]{HNST}.

\vskip1.5mm

Notice first that there are many  minimal surfaces in $\hip^2 \times \R$
 whose finite asymptotic boundary is the union of regular curves, see, for
instance, M. Rodriguez and F. Martin \cite{Ma-R} and the authors \cite{S-E}, \cite{SE-T1}.
 However, there are ``local obstructions'' to a curve be the asymptotic boundary
 of a minimal surface in $\Hip^2\times \R$, see \cite[Theorem 2.1]{SE-T2}.
  Also, B. Coskunuzer gave a necessary and sufficient condition on
a finite collection of Jordan curves in $\pain \hi2 \times  \R$ to be the
asymptotic
boundary of a complete area minimizing surface in $\hi2 \times  \R$,
\cite[Theorem 2.13]{Coskunuzer}. Afterward, B. Kloeckner and R. Mazzeo
generalized this result for a finite collection of Jordan curves in
$\pain (\hi2 \times  \R)$, \cite[Proposition 4.4]{K-M}.

\vskip1.5mm

 We wish next to describe briefly the behavior of certain minimal
surfaces in
$\Hip^2\times \R$,
   for their interesting properties  related to the results in this paper:
   First, we would like to summarize the
example
given by F. Morabito and M. Rodriguez \cite{Mo-R}: It  consists of a complete
minimal surface $M$ in
$\hip^2 \times \R$, invariant by a discrete   group of vertical translation.
The finite
asymptotic boundary is the finite union of vertical lines in
$\partial_{\infty}\hi2 \times \R.$ It is interesting to
note that {\em any}  nonempty domain  $S\subset M$ of finite
vertical  height has
{\em finite total curvature}. One can choose such $S$ to be a vertical graph. Of
course, the surface $M$ has infinite total
curvature, but the total curvature {\em does not concentrates}   in a
subset $S\subset M$ of bounded vertical height whose asymptotic boundary is a
vertical segment.  The reader is referred to the constructions due
to
L.  Hauswirth and A. Menezes \cite{H-M}, to find other related results.
\vskip1.5mm
 Secondly, let us consider now a classical minimal surface $M_1$ in
 $\hi2 \times \R$ which  has been useful  as  barrier in
many papers  about minimal graphs theory,
see  P. Collin and H. Rosenberg \cite[Lemma 1]{C-R}, B. Nelli and H.
Rosenberg \cite[Errata corrige 4 (b)]{N-R} and the authors
\cite[Theorem 4.1 (3)]{SE-T2}.
The surface $M_1$   is globally a vertical graph,
see \cite[Equation (32)]{S-E} for an explicit formula.
Thus, $M_1$ is stable.  The surface $M_1$
has been characterized by I. Fern\'andez and P. Mira \cite{F-M}.
A generalization of $M_1$ was carried out by the joint work of  J.
A. G\' alvez and H. Rosenberg \cite[Proposition 3.1]{G-R}, and by the
work of J. Plehnert \cite[Section 3.2]{P}.

 As a matter of fact, $M_1$ is invariant
under an one parameter group of hyperbolic translations along a
horizontal geodesic
 $\gamma \subset \hi2 \times \{0\}$.  Let us denote by
 $p_\infty, q_\infty$ the asymptotic
boundary of $\gamma$. Then the finite asymptotic boundary of $M_1$ is
composed of two vertical
half-lines $\partial_\infty\hi2 \times (0,+\infty)$,  issuing
from $p_\infty$ and $ q_\infty$,
and one of the two arcs of
$\partial_\infty\hi2 \times \{0\}\setminus \{p_\infty, q_\infty\}$,
see \cite[Proposition 2.1]{SE-T2}.
Let $S$ be a nonempty open subset  of $M_1$ whose asymptotic
boundary is
{\em a vertical segment} of the finite asymptotic boundary of $M_1.$
 Since $M_1$ is invariant by horizontal translations, it follows that
$S$ has infinite total curvature.

\vskip1.5mm
We observe that in $\hi2 \times \R$, finite total curvature of a complete oriented
minimal surface implies finite index. This is a theorem by P. B\'erard and the first author in
\cite{B-SE}. Notice  that in Euclidean space a famous result
of D. Fisher-Colbrie \cite{F-C} states that
a complete oriented minimal surface has finite index if and only if it has
finite total curvature.  Notice also that  finite index does not
imply finite total curvature in $\hi2 \times \R$,
as the preceding example shows. The catenoid in $\hip^2\times \R$  is another
counter-example: It has
infinite total curvature and index one \cite[Proposition 3.3 and Theorem
3.5]{B-SE}.

We pause momentarily to ask  here if the assumption ``complete'' can
be removed from the  B\'erard-Sa Earp  theorem ?  In $\R^3$,
S.-Y. Cheng and J. Tysk showed that a complete minimal surface with boundary
and with finite index has finite total curvature \cite[Theorem 5]{Cheng-Tysk}.
Afterward, A. Grigor'yan and S.-T. Yau  generalized this result assuming only
that $M$ is a minimal surface with finite index, that is they dropped the
assumption of completness \cite[Theorem 4.9]{Grig-Yau}.

\vskip1.5mm

 We observe that the two examples above show that when the finite asymptotic
boundary of a minimal surface in $\hi2 \times \R$ is a vertical
segment then the total curvature may be finite or infinite.

\subsection*{Acknowledgements}
{\Small   The second author wishes
to thank the {\em Departamento de Matem\'atica da
PUC-Rio} for their kind hospitality. }

\section{ Finite asymptotic boundary}\label{Sec.Asym}

\begin{definition}[Convergence to an asymptotic boundary point of $\h^2$]
\label{asymptotic convergence H2}
 Let $y_0\in \h^2$ be a fixed point of $\h^2$ and let
 $x_\infty \in \pain \h^2$. We denote by $[y_0, x_\infty) \subset \h^2$ the
geodesic ray issuing from  $y_0$ and with asymptotic boundary $x_\infty$. For
any $\rho >0$ we denote by $\gamma_\rho \subset \h^2$ the geodesic intersecting
the ray $[y_0, x_\infty)$ orthogonally at point $y_\rho$ such that
$d_{\h^2}(y_0,y_\rho)=\rho$. Let $\gamma_\rho^+$ be the component of
$\h^2 \setminus \gamma_\rho$ which contains $x_\infty$ in its asymptotic
boundary: $x_\infty \in \pain \gamma_\rho^+$.

Let $(x_n)$ be a sequence of points of $\h^2$. We say that
{\em $(x_n)$ converges to $x_\infty$}, denoted by $x_n \to x_\infty$, if for
any $\rho >0$ there exists $n_\rho \in \N$ such that $x_n \in  \gamma_\rho^+$
for any $n\geq n_\rho$.

\smallskip

We observe that if we choose the  Poincar\'e disc  model of $\h^2$,
then
$x_n \to x_\infty$ if and only if the sequence $(x_n)$ converges to $x_\infty$
in Euclidean sense.

Also, let us consider the Poincar\'e half-plane model of $\h^2$, then in this
model $\pain \h^2 = \R\cup \{ \infty \}$.
We have:
\begin{itemize}
 \item If $x_\infty \in \R$ then $x_n \to x_\infty$ if and only if the
sequence $(x_n)$ converges to $x_\infty$
in Euclidean sense.

\item If $x_\infty = \infty$ then $x_n \to x_\infty$ if and only if
$\vert x_n \vert \to +\infty$.
\end{itemize}

\end{definition}

\begin{definition}[Asymptotic boundary]\label{D.asymptotic}
$  $
\begin{enumerate}

\item We define the {\em asymptotic boundary} of $\hi2\times \R$ setting:
\begin{equation*}
 \partial_{\infty}(\h^2\times\R):=  \big(\partial_{\infty}\hi2 \times \R\big)
 \cup \big(\hi2 \times\{-\infty, + \infty\}    \big)
 \cup \big( \partial_{\infty} \hi2  \times\{-\infty, + \infty\}\big).
\end{equation*}
This decomposition means that  for a divergent sequence
$(p_n)$ of $\hi2 \times \R$  there are three possibilities for {\em converging
to infinity} (up to extracting a subsequence). That is, setting
$p_n=(x_n,t_n)\in \hi2\times \R$, we have the following cases:
\begin{itemize}
 \item $x_n \to x_\infty \in  \partial_{\infty}\hi2$
 (see Definition \ref{asymptotic convergence H2})  and
$t_n \to t_0\in \R$.
We say that $p_\infty:=(x_\infty,t_0)\in \partial_{\infty}\hi2 \times \R $ is
an asymptotic point {\em at finite height}.

\item $x_n\to x_0 \in \hi2$
and $t_n \to \pm \infty$. That  is $(p_n)$ converges to \newline
$p_\infty:=(x_0,\pm \infty)\in \hi2 \times\{-\infty, + \infty\}$.

\item  $x_n \to x_\infty \in  \partial_{\infty}\hi2$ and
$t_n \to \pm \infty$. That  is $(p_n)$ converges to \newline
$p_\infty:=(x_\infty,\pm \infty) \in \partial_{\infty} \hi2  \times
\{-\infty, +\infty\} $.
\end{itemize}

\vskip1.5mm

\item Let $\Omega \subset \hi2 \times \R $ be a nonempty subset.
We say that a  point
$p_\infty \in  \partial_{\infty}(\h^2\times\R)$ is an {\em asymptotic point} of
$\Omega$ if there is a sequence $(p_n)$ of $\Omega$ converging to $p_\infty$.

The set of asymptotic points of
$\Omega$, called the {\em asymptotic boundary of $\Omega$}, is denoted by
 $\partial_{\infty} \Omega$.

\vskip1.5mm

\item\label{item.finite} Let $\Omega \subset \hi2 \times \R $ be a
nonempty subset. The set of asymptotic points
  at finite height is called the {\em finite asymptotic boundary} and
  is denoted by $\partial_{\infty}^f \Omega$.

  The complement
$\pain \Omega \setminus \pain^f \Omega$ is called the {\em non finite
asymptotic boundary of $\Omega$}.

 We say that the finite asymptotic boundary $\partial_{\infty}^f \Omega$
 has {\em bounded vertical height} if
\begin{equation*}
  \exists t_1 >0,\, \partial_{\infty}^f \Omega \subset \{ (e^{i\, \theta}, t)
\in \partial_{\infty}\hi2 \times \R,\, |t|< t_1\}
\end{equation*}

\end{enumerate}
\end{definition}

\begin{definition}[Boundary of a surface]\label{D.boundary}
 Let $S\subset \hi2 \times \R$ be a surface.

\begin{enumerate}
 \item A point $p\in S$ is called an {\em interior point} of $S$ if there
 exists a proper embedding $Y : \di \rightarrow \hi2 \times \R$ such that
 $Y(\di)\subset S$ and $Y(0)=p$. The set of interior points of $S$ is denoted
 ${\rm int}\, (S)$.

\item  We define the {\em boundary} of $S$, denoted $\partial S$, as:
\begin{equation*}
 \partial S := \ov S \setminus {\rm int}\, (S),
\end{equation*}
where $\ov S$ is the closure of $S$ in $\hi2 \times \R$, we have
$\partial S \subset \hi2 \times \R$.
\end{enumerate}
\end{definition}

\vskip1.5mm

\section{Geometric Lemmas}

 The following result describes, in particular, the local behavior at infinity of
 a minimal surface  $M$ in
  $\h^2 \times \R$, whose finite asymptotic boundary  is an arc in
 $\pain^f M \setminus \pain (\partial M)$,
  which is
  not contained in a vertical line.

\begin{lemma}\label{L.Main Lemma}
Let $M $ be a connected  immersed minimal surface in $\h^2\times\R$. Assume
that:
\begin{enumerate}

\item\label{item.asymptotic boundary}   The finite asymptotic
boundary of $M$
is composed of an arc
$\alpha$  properly embedded in  $\pain \h^2 \times \R$.
\vskip1.5mm
\item\label{item.boundary surface}
There exists an open and simple arc $\alpha_0 \subset \alpha$ in
$\pain^f M \setminus \pain (\partial M)$ which is not contained in a vertical
line.
\end{enumerate}

\vskip1.5mm

Let $p_\infty :=(x_\infty, t_0)\in \pain \h^2 \times \R$, be any point of
$\alpha_0$ which does not belong to a vertical segment contained in
$\alpha_0$.

Then, for any $\,\varepsilon > 0$,
there exist a vertical plane $P_\varepsilon$ and
a component $P_\varepsilon^+$ of
$(\h^2\times\R) \setminus P_\varepsilon$,
such that, setting $S_\varepsilon:=M\cap P_\varepsilon^+$, we have

\vskip1.5mm

\claim \label{claim.1} $S_\varepsilon\subset \h^2 \times (t_0-\varepsilon, t_0
+\varepsilon)$,

\claim \label{claim.2}$\partial S_\varepsilon \subset P_\varepsilon$,

\claim \label{claim.3} {\em the asymptotic boundary of $S_\varepsilon$ is a
subarc
$\alpha_\varepsilon$ of $\alpha_0$ which is not contained in a vertical line:
$\pain S_\varepsilon=\alpha_\varepsilon \subset \alpha_0$,}

\claim \label{claim.4}{\em $p_\infty \in \alpha_\varepsilon$ and
$p_\infty \notin \pain P_\varepsilon$,}

\claim\label{claim.5} {\em $\pi (\alpha_\varepsilon)= \pi (\pain
P_\varepsilon^+)$, where
$\pi: \pain\h^2\times\R\goto \pain \h^2$ is the first projection.}

\claim  \label{item.limit} {\em Assume that $M$ is not contained in the slice
$\h^2 \times \{t_0\}$.
Then for any $\varepsilon >0$ there exists  $\varepsilon_0<\varepsilon$ such
that
for any $\varepsilon^\prime < \varepsilon_0$, $P_{\varepsilon^\prime}^+$ is
strictly
contained in $P_\varepsilon^+$. Hence $S_{\varepsilon^\prime}$ is strictly
contained in $S_\varepsilon$.  Furthermore
$\bigcap_{\varepsilon >0} P_\varepsilon^+=\emptyset$.

Observe that
$\pain S_\varepsilon = \pain^f S_\varepsilon$.}
\end{lemma}

\begin{proof}
 In the following we identify $\hi2 \times \{0\}$ with $\hi2$.

Let $p_\infty = (x_\infty, t_0) \in \alpha_0$  be a point as in
the statement.

 Let $y_0\in \hi2 $ be a fixed point.
We denote by $\gamma_0^+ \subset \h^2 $, the
geodesic ray issuing from  $y_0$ and with asymptotic boundary
$x_\infty$. 

For
any $\rho >0$ we denote by $\Pi_\rho \subset \h^2 \times \R$ the geodesic
vertical plane intersecting
the ray $\gamma_0^+$ orthogonally at point $y_\rho$ such that
$d_{\h^2}(y_0,y_\rho)=\rho$. Let $\Pi_\rho^+$ be the component of
$(\h^2 \times \R) \setminus \Pi_\rho$ which contains $x_\infty$ in its
asymptotic boundary, thus we have
$x_\infty \in \pain \Pi_\rho^+$ and $x_\infty \notin \pain \Pi_\rho$.

\smallskip

For any $\rho >0$ we denote by $\beta_\rho \subset \alpha_0$ the connected
component of $\alpha_0 \cap \pain \Pi_\rho^+$ containing $p_\infty$:
$p_\infty \in \beta_\rho \subset \alpha_0 \cap \pain \Pi_\rho^+$.

\smallskip

Let $\pi: \pain\h^2\times\R\goto \pain\h^2$ be the first projection.
 Recalling that $\alpha_0$ is properly embedded, it
 follows from \cite[Theorem 2.1]{SE-T2} that
 there
exists
$\rho_0 >0$ such that for any $\rho \geq \rho_0$ then $\pi (\beta_\rho)$ is
an
arc and
\begin{equation}\label{Eq.projection}
\begin{aligned}
&\text{\em for any $z_\infty \in \pi (\beta_\rho)$ its inverse
image by $\pi$ in $\beta_\rho$
is either a single point} \\
&\text{\em or a vertical segment}.
\end{aligned}
\end{equation}

We can also assume that $\rho_0$ is large enough so that
$M\cap\Pi_{\rho_0}\not= \emptyset$.

Denoting by $p_\rho^+, p_\rho^- \in \pain \Pi_\rho$ the two
 endpoints of
$\beta_\rho$, we have $\pi (p_\rho^+) \not= \pi(p_\rho^-)$. Therefore we get
$\pi  (\beta_\rho) =\pi (\pain \Pi_\rho^+)$ if $\rho \geq \rho_0$.

Observe that, by a continuity argument, for any $\varepsilon >0$ there is
$\rho_\varepsilon >\rho_0$ such that for any $\rho >  \rho_\varepsilon$ we
have
$\beta_\rho \subset \pain \h^2 \times [t_0 -\varepsilon, t_0 +\varepsilon]$.

Furthermore, since the finite asymptotic boundary of $M$ is a properly
embedded arc, if
$\varepsilon > 0$ is small enough we have
\begin{equation}\label{Eq.isole}
\beta_\rho = \partial_\infty ^f M \cap \pain \Pi_\rho^+
 \cap \big(\pain \h^2 \times [t_0 -\varepsilon,
t_0 +\varepsilon] \big),
\end{equation}
for any $\rho \geq \rho_\varepsilon$.

For any $\rho > \rho_0$, we denote by $M_\rho$ the union of the connected
components $M^\prime$ of $M\cap~ \Pi_\rho^+$ such that its finite asymptotic
boundary meets $\beta_\rho$, that is
$\pain^f M^\prime \cap \beta_\rho \not=\emptyset$. Therefore we have
\begin{itemize}
 \item $M_\rho \subset \Pi_\rho^+$,

\item  $\pain^f M_\rho = \beta_\rho$

\item $M_{\rho_2} \subset M_{\rho_1}$ if
$ \rho_0 <  \rho_1 <  \rho_2$.
\end{itemize}


We claim that there exists $\rho^\prime (\varepsilon) > \rho_\varepsilon$ such
that
\begin{equation*}
 M_{\rho^\prime (\varepsilon)}  \subset \h^2 \times
 (t_0-\varepsilon, t_0 + \varepsilon).
\end{equation*}

Indeed, otherwise there would exist a strictly increasing sequence
$(\rho_n)$ such that

\begin{itemize}

 \item $\rho_n > \rho _\varepsilon$ for any $n$ and $\rho_n \to +\infty$,

 \item for any $n$, 
 $M_{\rho_n}$ intersects
 $\big( \h^2 \times \{t_0+ \varepsilon \}\big) \cap \Pi_{\rho_n}^+$, or
 $\big( \h^2 \times \{t_0- \varepsilon \}\big) \cap \Pi_{\rho_n}^+$,
at some point $(y_n, t_0 \pm \varepsilon)$.

\end{itemize}

Observe that by construction we have $y_n \to x_\infty$.
Letting $n$ going to $+\infty$ we obtain that the asymptotic point
$(x_\infty, t_0 \pm \varepsilon)$ belongs to the finite asymptotic boundary
of $M$, which gives a contradiction  with (\ref{Eq.projection}) and
(\ref{Eq.isole}), with $\rho = \rho_\varepsilon$,
and the assumption that $p_\infty$ does not belong to a vertical segment
contained in $\alpha_0$.

Now we set
\begin{equation*}
 P_\varepsilon:= \Pi_{\rho^\prime(\varepsilon)},\  S_\varepsilon:=
 M_{\rho^\prime (\varepsilon)},\ 
 \text{and} \ \alpha_\varepsilon:= \beta_{\rho^\prime(\varepsilon)},
\end{equation*}
We have just seen that
\begin{itemize}
\item  $S_\varepsilon \subset \h^2 \times
 (t_0 -\varepsilon, t_0 +\varepsilon)$,

\item  $\pain^f S_\varepsilon = \alpha_\varepsilon$,

\item $\pi (\alpha_\varepsilon)=\pi (P_\varepsilon^+ )$,
\end{itemize}
therefore we have $\pain S_\varepsilon= \alpha_\varepsilon$.
 Since $p_\infty$ does not belong to the asymptotic boundary of
$\partial M$, we can choose $\rho_\varepsilon$ so large that for any
$\rho > \rho_\varepsilon$ we have

\begin{equation*}
 \partial M \cap \Pi_\rho^+ \cap
 \Big( \h^2 \times
 (t_0 -\varepsilon, t_0 +\varepsilon)\Big) =\emptyset.
\end{equation*}

Therefore we get that $\partial S_\varepsilon \cap P_\varepsilon^+=\emptyset$
and then $\partial S_\varepsilon \subset P_\varepsilon$. This proves
Claims \ref{claim.1}-\ref{claim.5}.

\medskip

Now we suppose that $M$ is not contained in  $\h^2 \times \{t_0\}$.

Let $\varepsilon >0$ be such that Claim (\ref{item.limit}) does not hold.
Then there exists
a strictly decreasing positive sequence $(\varepsilon_n)$ such that $\varepsilon_n \to 0$ and
$P_{\varepsilon_n}^+$ is not contained in $P_\varepsilon^+$.
 Recall that, for any $n$, we have
$P_{\varepsilon_n}^+ = \Pi_{\rho_n}^+$ for some $\rho_n >0$. Terefore, if
$P_{\varepsilon_n}^+$ is not contained in $P_\varepsilon^+$ we obtain that
 $P_\varepsilon^+$ is contained in $P_{\varepsilon_n}^+$, for any $n$.
Thus $S_{\varepsilon}\subset S_{\varepsilon_n}$,
and consequently $S_\varepsilon \subset \h^2 \times [t_0-\varepsilon_n, t_0+\varepsilon_n]$ for
any $n$.
Letting $n$ going to $+\infty$ we get $S_\varepsilon\subset \h^2 \times \{t_0\}$. By the
 analytic continuation property,  we get that  $M\subset \h^2 \times \{t_0\}$,
  which leads to a contradiction.

  Notice that the same argument shows that,
if $M$ is not contained in  $\h^2 \times \{t_0\}$,
  $P_\varepsilon^+$ goes to infinity as $\varepsilon$ goes to zero, that is
  $\rho^\prime (\varepsilon) \to +\infty$ if $\varepsilon \to 0$.
   Therefore we
get that $\bigcap_{\varepsilon >0} P_\varepsilon^+ = \emptyset$. This
accomplishes the proof of the Lemma.
\end{proof}

\

In our context, it is natural to expect that the area of a minimal surface $M$ in
$\h^2 \times \R$ is infinite. More precisely, we derive the following result.

\

\begin{lemma}\label{Area Lemma}
Let $M $ be a minimal surface immersed  into $\h^2\times\R$. Assume that the finite
asymptotic boundary of the boundary of $M$ is not equal to the  finite
asymptotic boundary  of the boundary of $M$, that is
  $\partial_\infty^f (\partial M) \not= \partial_\infty^f M$.
  Then $M$ has infinite area.

\end{lemma}

\begin{proof}
By assumption there exists a  finite asymptotic point $p_\infty$ of $M$ which
is  not an asymptotic point of the boundary of $M$:
$p\in \partial_\infty^f M \setminus \partial_\infty^f (\partial M)$.

Let $(p_n)$ be a sequence of points of $M$ which converges to $p_\infty$,
  see Definition \ref{D.asymptotic}.
 Let $\delta >0$ be  a fixed real number.
Then, since $p_\infty$ is not an asymptotic boundary point of $\partial M$,
there exists $n_0 \in \N$ such that $d_M(p_n, \partial M) >2 \delta$ for any
$n \geq n_0$, where $d_M$ means the intrinsic distance on $M$.

\smallskip

For each $n$, let $D(p_n, \delta) \subset M$ be the geodesic disc on $M$
centered at $p_n$ and with radius $\delta$. Then for any $n \geq n_0$ we have
$D(p_n, \delta) \cap \partial M =\emptyset$.

Furthermore, up to extracting a subsequence, we can assume that there exists
$n_1 \in \N$, $n_1 > n_0$, such that
$D(p_n, \delta) \cap D(p_m, \delta) =\emptyset$ for any $m,n \geq n_1$.

In an other hand, a result of K. Frensel \cite[Theorem 3 and Remark 4]{Frensel}
states that there exists a fixed real
number
$\alpha >0$ such that $\text{Area}\,( D(p_n, \delta)) >\alpha$ for any $n \geq n_1$.
We conclude that $M$ has infinite area.
\end{proof}

\begin{remark}
 The assumption on the asymptotic boundary in Lemma \ref{Area Lemma} is
crucial as we can see from the following examples.
\begin{enumerate}
 \item A geodesic triangle in $\h^2$ with one, or more, vertices
in the asymptotic boundary of $\h^2$ has finite area
\cite[Lemme 2.5.23 and Th\'eor\`eme. 2.5.24]{Livre}.
We observe that
the asymptotic boundary of the triangle is equal to
the asymptotic boundary of its boundary.

\smallskip

\item We can construct
a domain in $\h^2$ with finite area and whose the asymptotic boundary is the
whole $\pain \h^2$.

\smallskip

Indeed, consider the Poincar\'e disc model of $\h^2$.
For any $\rho >0$ we denote by $S_\rho \subset \h^2$ the circle centered at $0$
with radius $\rho$.

Let $(\rho_n)$ be a strictly increasing  sequence of positive real numbers such
that $\rho_n \to +\infty$. Now we consider another sequence of positive real
numbers $(\rho_n^\prime)$, $\rho_n^\prime >\rho_n$, such that, calling $A_n$
the open annulus bounded by the circles $S_{\rho_n}$ and  $S_{\rho_n^\prime}$,
we have:
\begin{itemize}
 \item the annuli $A_n$ are mutually disjoint,

 \item $\text{Area}\, (A_n) < \frac{1}{n^2}$  for any $n\in \Ne$.
\end{itemize}
Let $y_0 \in \h^2$, $y_0\not=0$, be any point on the imaginary axis such that
its hyperbolic distance to $0$ is lesser that $\rho_1/2$. We call $T$ the
open geodesic triangle with vertices $0, y_0$ and $1$, observe
that this last vertex is the unique vertex
of $T$ belonging to the asymptotic boundary of $\h^2$.

 Then we set $U:= T\cup \bigcup_{n\geq 1} A_n$. By construction $U$
is a domain of $\h^2$ satisfying:
\begin{itemize}
\item $\text{Area}\, (U)$ is finite, (since $\text{Area }(T)$ is finite
\cite[Lemme 2.5.23]{Livre}).

\item $\pain U= \pain \h^2$ and  also $\pain (\partial U)= \pain \h^2$. In
particular the asymptotic boundary of $U$ is equal to the asymptotic
boundary of its boundary.
\end{itemize}

\smallskip

Observe that the domain $U$ is infinitely connected. We can modify
slightly $U$ in
order to obtain a simply connected domain. For that we consider a fixed point
$x_0 \in \h^2$, $x_0 \not= 0$, on the real axis such that its hyperbolic
distance to $0$ is lesser that $\rho_1/2$. We call $\Sigma^+$
(resp. $\Sigma^-$) the closed geodesic triangle in $\h^2$ with vertices
$0, x_0$ and $i$ (resp. $0, x_0$ and $-i$). We set
\begin{equation*}
 \Omega := T \cup \big(\bigcup_{k\geq 1} A_{2k}\setminus \Sigma^+\big)
\cup
 \big(\bigcup_{k\geq 1} A_{2k-1}\setminus \Sigma^-\big).
\end{equation*}
We have by construction $\Omega \subset U$, therefore $\Omega$ has finite area.
Furthermore:
\begin{itemize}
\item $\Omega$ is a simply connected domain,

\item  $\pain \Omega= \pain \h^2$ and  also
$\pain (\partial \Omega)= \pain \h^2$. In
particular the asymptotic boundary of $\Omega$ is equal to the asymptotic
boundary of its boundary.
\end{itemize}

\end{enumerate}

\end{remark}

\section{Main Theorem}\label{M. theorem}

\begin{theorem}\label{Main theorem}

Let $M $ be a connected and  oriented  minimal surface immersed in
$\h^2\times~\R$
$($not necessarily complete$)$. As in \lemref{L.Main Lemma} we assume that
 \begin{enumerate}

\item
The finite asymptotic boundary of $M$
is composed of an arc
$\alpha$  properly embedded in  $\pain \h^2 \times \R$.

\vskip1.5mm
\item
There exists an open and simple arc $\alpha_0 \subset \alpha$ in
$\pain^f M \setminus \pain (\partial M)$ which is not contained in a vertical
line.
\end{enumerate}
 Assume further that $M$ is stable.
 Then $M$ has infinite total curvature.

 \smallskip

  Moreover, let $p_\infty :=(x_\infty, t_0)\in \pain \h^2 \times \R$, be any
point of
$\alpha_0$ which does not belong to a vertical segment contained in
$\alpha_0$. Then
$\vert n_3(p)\vert \to 1$ if
$p \to p_\infty$, $p\in M$, where  $n_3$ is the third coordinate of the
Gauss map of $M$.

\end{theorem}

\begin{proof}

Observe that, taking into account Lemma \ref{Area Lemma}, if $M$ is contained
in a slice $\hi2 \times \{t_0\}$ then there is nothing to prove. Thus from now
we assume that $M$ is not contained in a slice.

We recall the Gauss equation of the immersion (see \cite[Lemma 4]{HST}):
\begin{equation}\label{Eq.Gauss}
K =K_{\text{ext}}  - n_3^2,
\end{equation}

where $K$ is the Gaussian
curvature and $K_{\text{ext}}$ is the extrinsic curvature.

Since $M$ is a minimal surface, we have $K \leq -n_3^2$.

\

By assumption, the finite asymptotic boundary of $M$ is an arc $\alpha$, and
there exists a simple arc
 $\alpha_0 \subset \pain^f M \setminus \pain (\partial M)$,
 $\alpha_0 \subset \alpha$,
which is not contained in a vertical line.
  Let $p_\infty :=(x_\infty, t_0) \in  \alpha_0$ as
in the statement.

For any $\varepsilon >0$ we consider the minimal surface
$S_\varepsilon \subset M$ given by
\lemref{L.Main Lemma}.

\

\noindent {\bf Claim} {\em For any real number  $c\in (0,1)$,
there exists
$\varepsilon > 0$ such that $\vert n_3 (p) \vert > c$
for any $p\in S_\varepsilon$. Consequently,
\begin{equation} \label{Eq.n3}
\vert n_3 (p) \vert \to 1, \ \text{if}\ p \to p_\infty,\ p\in M.
\end{equation}
}

Let us assume momentarily that the Claim holds.

Using  the Claim above and the Gauss equation (\ref{Eq.Gauss}), we have
\begin{equation*}
\int_M K\, dA \leq \int_{ S_\varepsilon} K\, dA \leq -c^2 \text{Area} (S_\varepsilon).
\end{equation*}
By combining with Lemma \ref{Area Lemma}, we deduce therefore
that $M$ has
infinite total curvature, as desired. Thus it remains to prove the Claim.

\

\noindent {\em Proof of the Claim.}

Assume, by contradiction, that the Claim does not hold. Then, there exists a
fixed number $c\in (0,1)$ such that for any $n \in \Ne$ there is a point
$p_n \in S_{1/n}$ satisfying
\begin{equation}\label{Eq.constante}
\vert n_3 (p_n)\vert \leq c
\end{equation}
It follows from Lemma \ref{L.Main Lemma}-(\ref{item.limit})
(or from its proof),
  that $p_n \to p_\infty$ as $n\to \infty$, see Definition \ref{D.asymptotic}.

\smallskip

 Let $n_0 \in \Ne$ be a  positive integer. We have $p_n \in S_{\frac{1}{n_0}}$
 for any integer $n \geq n_0$  large enough.  Therefore, up to
extracting a subsequence, we
 can assume that   for any  $ n \in \Ne$, $n > n_0$, we have
 $d_M (p_n, \partial S_{\frac{1}{n_0}}) > 1$,
where $d_M$ is the intrinsic
 distance on $M$.

\smallskip

From now on we consider the Poincar\'e disc model of $\h^2$. Letting
$p_n := (x_n, t_n) \in \h^2 \times \R$, for any $n> n_0$ we denote by $T_n$
the  hyperbolic translation on $\h^2$ along the geodesic passing through $x_n$
and $0$, such that $T_n (x_n)=0$. We also denote by $T_n$ the horizontal
translation of $\h^2 \times \R$ induced by this isometry of $\h^2$.

\smallskip

Now we proceed as in the proof of \cite[Theorem 2.5]{SE-T3}.

Observe that for any $n >n_0$ the translated surface
$T_n \left(S_{\frac{1}{n_0}}\right)$ is stable and oriented. We deduce from
\cite[Main Theorem]{RST} that, far away from the boundary,
 we have  uniform a priori upper estimates
of the norm of the second fundamental form of
$T_n\left(S_{\frac{1}{n_0}}\right)$.

\smallskip

We consider $\hi 2\times\R$ as an open set of Euclidean space $\R^3$.
We deduce from  \cite[Proposition 2.3]{RST} and from
\cite[Proposition A.1]{SE-T3},
that there exists a real number
$\delta >0$, which does not depend on $n$, or on $n_0$,
such that for any
$n>n_0$, a part $ \Sigma_{n}$ of
$T_n \left(S_{\frac{1}{n_0}}\right)$
is the Euclidean graph
of a function $u$ defined on the
disc centered at point $T_n(p_n)$ with Euclidean radius $\delta$ in the
tangent plane of $ \Sigma_{n}$ at $T_n(p_n)$.
Furthermore, the norm of the Euclidean
gradient of the
function $u$ is bounded above by $1$.

\smallskip

As a matter of fact, from the discussion
after the proof of Lemma 2.4 in \cite{C-M}, we get the following.
\begin{align}\label{Eq.Fact}
\text{Fact: }& \text{\em
for any
$r \in (0,1)$ there exists $\delta (r) \in (0,\delta)$ such that the norm of
the gradient} \notag \\
&\text{\em of the function $u$ is bounded above by $r$  on the disc of Euclidean
radius $\delta (r)$  }
\end{align}
Observe that we can use \cite[Lemma 2.4]{C-M} since we have
a priori estimates for the norm of the Euclidean second fundamental form.
Those estimates follow from \cite[Proposition A.1]{SE-T3}.

\medskip

Observe that, since
$d_M \left(T_n(p_n), \partial T_n\left(S_{\frac{1}{n_0}}\right)\right) > 1$
for any $n > n_0 >0$, the constant $\delta$ can be chosen so that
$\Sigma_n \cap \partial T_n\left(S_{\frac{1}{n_0}}\right)=\emptyset$.

\smallskip

Let  $\nu_n$ be the unitary normal along $T_n\left(S_{\frac{1}{n_0}}\right)$
in the Euclidean metric.
We denote by  $\nu_{n,3}$ the vertical component of $\nu_n$. Recall that
$|n_3 (p_n)|< c$ for any $n >n_0>0$,    see (\ref{Eq.constante}).
Comparing the product metric  of $\hi2 \times\R$ with the Euclidean metric,
it can be shown that there exists $c^\prime \in (0,1)$, which does not
depend on $n$ or $n_0$, such that
$|\nu_{n,3}   (T(p_n))|< c^\prime$ for any $n>n_0>0$,
(see the formula of the unit normal vector field of a vertical graph in  the
proof of \cite[Proposition 3.2]{SE-T4}).

\bigskip

This implies that the tangent planes of $\Sigma_n$ at points $T(p_n)$ have
 a Euclidean slope
bounded from below  uniformly (with respect to $n>n_0$).

We recall that $S_{\frac{1}{n_0}}\subset \h^2 \times
 (t_0- \frac{1}{n_0}, t_0 + \frac{1}{n_0}) $, see Lemma
\ref{L.Main Lemma}, thus the same occurs
 for any $\Sigma_n$. We infer therefore a contradiction with the fact
(\ref{Eq.Fact}) above
 since then, for $n_0$ large enough and
 $n > n_0$, the surface $\Sigma_n$ would intersect
  $\h^2 \times \{t_0 \pm \frac{1}{n_0}\}$.
\end{proof}




\bigskip

\begin{remark}
 As a matter of fact, in Theorem \ref{Main theorem} the stability assumption is
only used to ensure a priori estimates for the second fundamental form of $M$.
We think that stability is a hypothesis  simpler to handle  than bounded
second fundamental form since, for example,
any vertical or horizontal minimal graph is stable.
\end{remark}

\bigskip

\begin{remark}
Given a bounded function g on
 $\partial_\infty \hip^2 \times \{0\}$,
continuous except perhaps at a finite set of points $S,$ there exists
 a minimal entire extension $u$ of $g$ \cite[Corollary 4.1, Remark 4
(2)]{SE-T2}.
We remark that the problem of Dirichlet at infinity ($g$ is continuous) was
solved by B. Nelli and H. Rosenberg \cite{N-R0}, \cite{N-R}.

\thmref{M. theorem} ensures that  these entire
graphs have infinite total curvature.

 However, the fact that all non trivial ($g\not\equiv \text{cst}$)
such
entire graphs have infinite total curvature, follows directly from Huber theorem
\cite[Theorem 15]{Hu}, see also \cite[Th\'eor\`eme 2.~4.~10]{eBook}:  In
fact, a complete simply connected minimal
surface
immersed into $\hip^2 \times \R$ of finite  total curvature  is conformally
equivalent to $\Cx$.

 On the other hand, it is well-known that the height function of a
minimal surface $M$ conformally immersed into $\hip^2 \times \R$ is a  harmonic
function on $M,$ see for instance \cite[
Proposition 7]{SE-T1}.  Thus, by combining this  two facts we derive that the
finite total curvature assumption  would lead to a contradiction, because there
is no non constant bounded harmonic
function over $\Cx.$
\end{remark}

\begin{corollary}\label{C.arc}
Let $M$ be a minimal graph in $\hip^2 \times \R$
such that its finite asymptotic boundary
is a
graph over an  arc of $\partial_{\infty}\hi2 \times \{0\}$ and
is different from the asymptotic boundary of $\partial M$.
Then $M$ has infinite total curvature.

Furthermore, for any interior point
$p_\infty$ of $ \pain^f M $ such that
$p_\infty \notin  \pain^f (\partial M)$, we have
$\vert n_3(p)\vert \to 1$ if
$p \to p_\infty$, $p\in M$.

\end{corollary}

\begin{corollary}\label{C. end}

Let $M$ be an oriented stable minimal surface immersed into $\hip^2 \times \R$
with compact boundary (e.g. a minimal graph with compact boundary),
whose  asymptotic boundary
is a  $($continuous$)$
graph over  the whole $\partial_{\infty}\hi2 \times \{0\}$. Then $M$ has infinite total curvature.

Furthermore,  if $p_\infty\in  \pain M $,
 we have $\vert n_3(p)\vert \to 1$ if
$p \to p_\infty$, $p\in M$.
\end{corollary}

\begin{remark}
 The above Corollary applies to a minimal graph with compact boundary whose
asymptotic
boundary is a graph over $\partial_{\infty}\hi2 \times \{0\}$. We refer to
\cite[Theorem 5.1]{SE-T2} for a
existence result of such graphs, when the boundary is a Jordan curve
$C\subset \hip^2 \times \{0\}$ satisfying an ``exterior circle
of radius $\rho$ condition'' .  So, all of these examples have infinite
(intrinsic) total curvature, by applying the  \coref{C. end}.

 In the classical case of the end of  a catenoid,   the above
Corollary
follows
from an explicit computation carry out in \cite[Proof of the Proposition
3.3]{B-SE}.
\end{remark}

\section{A  particular example}\label{remarkable}

We consider the Poincar\'e disc model of $\h^2$. Let
$\theta_0 \in (0,\pi/2)$ be a fixed number.
Let $\gamma \subset \h^2$ be the geodesic with asymptotic boundary
$\{1, e^{i \theta_0}\}$.

 Let $U\subset \h^2$ be the domain whose boundary is the union of the geodesic
rays $[0, 1)$ and $[0,i)$ with the geodesic $\gamma$, whose asymptotic
boundary is the asymptotic arc
$\gamma_\infty:= \{e^{i\theta},\ \theta_0 \leq \theta \leq \pi/2 \}$ of
$\pain \h^2$ with the point $1$.

Let $A >\pi$ be a real number to be chosen  later. We consider the Dirichlet
problem $(P)$ on $U$ with boundary data
\begin{itemize}
 \item $0$ on the geodesic rays $[0, 1)$ and $[0,i)$,

 \item $-A$ on the asymptotic arc $\gamma_\infty$,

\item $+\infty$ on the geodesic $\gamma$.
\end{itemize}

Using \cite[Theorem 4.1]{SE-T2}  and the minimal surface $M_1$
described
in \cite[proposition 2.1 (2)]{SE-T2}, as in \cite[Example 4.1]{SE-T2} or as in
\cite[Theorem 5.1 (n=2)]{SE-T4},
we can solve the Dirichlet problem $(P)$ above
and find a solution $g: U\rightarrow \R$ whose finite asymptotic boundary of
the graph, $M$, is
\begin{equation*}
 \Big(\{i\} \times [-A,0] \Big) \cup \Big(\gamma_\infty \times \{-A \} \Big)
\cup
 \Big( \{e^{i \theta_0}\} \times [-A,+\infty) \Big)
 \cup  \Big( \{1\} \times [0,+\infty) \Big)
\end{equation*}
and the non finite asymptotic boundary of $M$ is
\begin{equation*}
 \Big(\gamma \times \{ +\infty\}\Big)
 \cup  \Big( \{1, e^{i \theta_0}\}\times \{ +\infty\}\Big).
\end{equation*}

\

\setcounter{claim}{0}

We claim that the following phenomena hold.

\claim \label{claim.infinite} Let $p_\infty, q_\infty$ be points in
 $\gamma_\infty$ such that $p_\infty, q_\infty \not= e^{i \theta_0}, i$.
 Let $\alpha\subset \h^2$ be the geodesic whose asymptotic boundary is
 $\{ p_\infty, q_\infty  \}$. We call $U_1 \subset \h^2$ the component of
 $\h^2 \setminus \alpha$ whose asymptotic boundary is the subarc
 $[p_\infty, q_\infty]$ of $\gamma_\infty$. We have $U_1 \subset U$. Let
 $S_1\subset M$ be the graph of $g$ restrited to $U_1$.

 Then it follows
 from Corollary \ref{C.arc} that  $S_1$ has infinite total
curvature.

\smallskip

Furthermore,  we have
$\vert n_3(q)\vert \to 1$ if
$q \to   \gamma_\infty \setminus \{ e^{i \theta_0}, i\}$,
$q\in U_1$.

\medskip

\claim \label{claim.vertical} Let $S_2 \subset M$ be a domain such that
 its asymptotic boundary is a  compact arc of
 $ \Big( \{1\} \times [0,+\infty)  \Big)$. Then it can be showed that $S_2$ has
finite total curvature.

\

In order to outline the proof of Claim (\ref{claim.vertical}), we state the
following facts.

\smallskip

\begin{enumerate}
 \item \label{item.domain} {\em Let $V\subset U$ be a subdomain
such that  its asymptotic boundary is constituted of zero, one or two points of
$\pain U$. If $g$ is constant along the boundary of $V$, then $g$ is constant
on $V$, which leads to a contradiction. }

To prove this fact we use the maximum principle and the family of complete
minimal surfaces $M_d$, $d>1$,
described
in \cite[Proposition 2.1-(1)]{SE-T2} and in the
proof of \cite[Theorem 3.2]{NST}.

\vskip1.5mm

\item\label{item.critical} {\em The function $g$ has no critical points on $U$.
}

Indeed, if $g$ would have a critical point $p\in U$, with $g(p)=c >-A$, then in
a neighborhhod of $p$, the level set $g^{-1}(\{c\})$
is constituted at least of four analytic arcs issuing from $p$. Observe that
any level set of $g$ cannot have end points in $U$.
 Observe also that the asymptotic boundary of the level set
$g^{-1}(\{c\})$ is included in $\{1, e^{i\theta_0},i\}$.
Therefore, continuing any
of the  analytic arcs issuing from $p$, we obtain a domain $V$ as in item
(\ref{item.domain}), which leads to a contradiction.

\vskip1.5mm

\item\label{item.level} {\em For any real number $c\in (-A, +\infty)$,
 the level set
$g^{-1}(\{ c \})$ is constituted of an unique simple divergent curve in $U$
and its asymptotic boundary is contained in $\{1, e^{i\theta_0}, i   \}$.
}

To proof this assertion, we first study the different possible cases of the
level set $g^{-1}(\{ 0 \})$. Then for each one of those cases we apply
 the items
(\ref{item.domain}) and (\ref{item.critical}).

\vskip1.5mm

\item {\em Using the reflection principle along the two geodesic rays starting
from the origine and whose asymptotic boundary are $\{ 1\}$ and $\{i\}$
respectively, we
obtain a complete minimal surface $\wt M \subset \h^2 \times \R$ which is a
graph. Hence $\wt M$ is stable and from \cite[Main Theorem]{RST} we obtain
global upper estimates for the norm of the second fundamental form of $M$.
 Observe that those upper estimates do not depend on A.}

\vskip1.5mm

We denote by $(0,1)$ the  open geodesic ray starting at $0$
whose asymptotic boundary
is $1$. We denote by $R$ the reflection in $\h^2 \times \R$
 with respect to the
geodesic ray $(0,1)$ and we set $M^*:= M \cup (0,1) \cup R (M)$. Then $M^*$ is
a minimal surface which is a graph over the domain
$U_1 := U \cup (0,1) \cup R (U)$ of $\di$.
\vskip1.5mm

\item\label{item.rho} {\em For any $\rho >0$ we set
$\mathcal{Z}_\rho = \{ \xi \in U_1,\ d_{\h^2}(\xi, \gamma) < \rho \}$.
Then, for any $c\in (0,1)$ there exists $\rho_c >0$ such that
$\vert n_3 (\xi)\vert < c$, for any $\xi \in \mathcal{Z}_{\rho_c}$.
Furthermore the number $\rho_c$ does not depend on $A$.
}

Indeed, if the assertion is not true, there would exist a sequence
$(p_n)$ in $U_1$ such that
\begin{itemize}
 \item $d_{\h^2}(p_n, \gamma) \to 0$,

\item $\vert n_3 (p_n) \vert \geq c$ for any $n$.
\end{itemize}

\medskip

Let $\xi_0 \in \gamma$ be any fixed point,  we set
$D_1:= \{ \xi \in U,\ d_{\h^2}(\xi, \xi_0) <1 \}$.

Observe that for any  $n$  large enough we can use a translation
$T_n$ along the geodesic
$\gamma$ to send $p_n$ on  a  point $T_n (p_n)$ in the domain $D_1$. By
construction we have that $d_{\h^2}(T_n (p_n), \gamma) \to 0$.
 Using the global upper estimates for the norm of the second
fundamental form of $M$, we  can
proceed as in the proof of the Claim in Theorem \ref{Main theorem} to reach a
contradiction.
 Since those upper estimates do not depend on $A$, we obtain
also that the number $\rho_c$ does not depend on A.

\vskip1.5mm

\item\label{item.rectangle} {\em
 From now on, we choose a fixed number
 $c\in (0,1)$. Let
$\rho_c >0$ be the positive real number given in item $(\ref{item.rho})$.
We call $\alpha \in U$ the geodesic whose asymptotic boundary is
$\{i, e^{i\theta_0}\}$. We denote by $M_{d_A}$,
 $d_A >1$, the surface of the family
$M_d$, described
in \cite[Proposition 2.1-(1)]{SE-T2} and in the
proof of \cite[Theorem 3.2]{NST},  such that
\begin{itemize}
 \item the height of $M_{d_A}$ is $A$,

\item  $M_{d_A}$ is symmetric with respect to the slice
 $\h^2 \times \{0\}$,

\item for any $t\in (-\frac{A}{2}, \frac{A}{2})$ the intersection
$M_{d_A} \cap \big(\h^2 \times \{ t  \}\big)$ is an equidistant curve of the
geodesic $\alpha$.
\end{itemize}
Then, we have $M \cap M_{d_A} = \emptyset$. Consequently, we have
$M\cap \left( M_{d_A} + (0,0,t)  \right)=\emptyset$ for any
$t \geq 0$.
}

Observe that, using the notations of \cite[Proposition 2.1-(1)]{SE-T2} we have
$A=2 H(d_A)$. Moreover the asymptotic boundary of $M_{d_A}$ is
\begin{equation*}
 \left(\{i, e^{i\theta_0} \} \times \Big[-\frac{A}{2}, \frac{A}{2}\Big]\right)
\cup
\left( \gamma_\infty \times \Big\{ -\frac{A}{2}, \frac{A}{2} \Big\} \right) .
\end{equation*}
where $\gamma_\infty:= \{e^{i\theta},\ \theta_0 \leq \theta
\leq \pi/2 \}\subset \pain \h^2$.

Let $\theta_1 \in (\theta_0, \pi/2)$ be a fixed number and let
$\beta \subset \h^2$ be the geodesic whose asymptotic boundary is
$\{-e^{i\theta_1}, e^{i\theta_1}  \}$.

To prove the first assertion we consider the hyperbolic translation along the
geodesic $\beta$ and proceed as in \cite[Theorem 3.2]{NST}. The second
assertion is a consequence of the first one.

\vskip1.5mm

\item\label{item.ray} {\em Let $\delta_0 \subset \h^2$ be the
geodesic ray
issuing from 0 and
with asymptotic boundary $\{ e^{i\theta_0}  \}$.
For any $r>0$ we denote by $Q_r$ the vertical geodesic plane intersecting
orthogonally $\delta_0 $ at distance $r$ from 0. Let
$Q_r^+ \subset \h^2 \times \R$ be the component of
$(\h^2 \times \R) \setminus Q_r$ containing $e^{i\theta_0}$ in its asymptotic
boundary.

Let $c\in (0,1)$  be a fixed number and let
$\rho_c >0$ be the positive real number given in item $(\ref{item.rho})$.

Then, if $A$ is large enough, there exists $r>0$
so that
\begin{equation}\label{Eq.inclus}
 \left(M\cap Q_r^+\right) \cap \{t\geq 0\} \subset
\mathcal{Z}_{\rho_c} \times [0, +\infty).
\end{equation}
}

The proof of the assertion is based upon the following observation.

Since $\rho_c >0$ does not depend on $A$, observe that for $A>0$ large enough
we have
$M_{d_A} \cap \big(\mathcal{Z}_{\rho_c} \times [0, +\infty)\big) \not=
\emptyset$. For such a number $A$, using the last affirmation of item
(\ref{item.rectangle}) certainly we can find a number
$r>0$ large enough satisfying (\ref{Eq.inclus}).

\bigskip

Now we pause momentarily to recall some facts derived from \cite{HNST}, \cite{HST} and
\cite{SE-T3}.

Let $X : \di \rightarrow \h^2 \times \R$ be a conformal parametrization of $M$.
We set as in \cite{SE-T3} $X=(F,h)$, thus
 $F : (\di, g_{\rm euc}) \rightarrow \h^2$ is a
harmonic map and $h : \di \rightarrow \R$ is a harmonic function,
 where $g_{\rm euc}$ is the Euclidean metric.

Since $g$ has no critical point on $U$, we have $\vert n_3 \vert \not= 1$ along
$M$. Therefore we can define a  real function $\omega$ on $\di$
by the relation:
$ \tanh \omega = n_3$.

We consider also the function $\phi$ on $\di$ defined by
$\phi := (\sigma \circ F)F_z \ov F_z$, where $\sigma$ is the conformal factor
of the hyperbolic metric of $\h^2$. Since $F$ is a harmonic map,
$\phi$ is a holomorphic function.

The metric induced on $\di$ by the immersion $X$ is
\begin{equation*}
  ds^2=4 \cosh^2 (\omega)\, |\phi |\, |dz|^2.
\end{equation*}
Moreover we have $\phi(z)=-(h_z (z))^2$, \cite[Proposition 1]{SE-T1}.
Now we define a holomorphic function $W$ on $\di$ setting
\begin{equation*}
 W(z)= \int \sqrt{\phi (z)}\, dz,
\end{equation*}
where the square root of $\phi$ is chosen so that
\begin{equation}\label{Eq.height}
 h=2 \Imag W(z)
\end{equation}

\vskip1.5mm

\item\label{item.univalent} {\em The function $W$ is a univalent map, hence $W$
is a
holomorphic diffeomorphism between
$\di$ and the open subset $\wt \Omega := W(\di)$ of $\C$.
}

It follows from item (\ref{item.level}) that for any $c\in (-A,+\infty)$,
the level curve
$h^{-1}(\{ c \})$ is constituted of an unique simple divergent curve in $\di$.
We deduce from item (\ref{item.critical}) that $h$ has no critical point.
Consequently the conjugate function $^*h$ is strictly monotonous along any
level curve of $h$. Combining with Formula (\ref{Eq.height}) we conclude   that
$W$ is an univalent map.

\vskip1.5mm

Now we define the function $\wt \omega$ on $\wt \Omega$ setting
\begin{equation}\label{Eq.tilde}
 \wt \omega := \omega \circ W^{-1}.
\end{equation}
 We know from \cite[Formula (12)]{HNST} that the function
$\wt \omega$ satisfies
\begin{equation}\label{Eq.Laplacian}
 \Delta \wt \omega = 2 \sh \wt \omega,
\end{equation}
where $\Delta$ is the Laplacian for the Euclidean metric.

We consider also the new conformal parametrization
$\wt X : \wt \Omega \rightarrow \h^2 \times \R $
of $M$ given by
$\wt X= X \circ  W^{-1}$. Denoting by $w$ the coordinate on $\wt \Omega$, the
induced metric on $\wt \Omega$ reads  as
\begin{equation}\label{Eq.metric}
 d\wt s^2 = 4\cosh^2 \wt \omega (w)\, \arrowvert dw\arrowvert^2,
\end{equation}
We define also the function $W_0 : U \rightarrow \wt \Omega \subset \C$ setting
 $W_0 : = W \circ (\Pi \circ X)^{-1}$, where
$\Pi : \h^2 \times \R \rightarrow \h^2$ is the first projection. As a matter of
fact, $W_0$ is nothing but the function $W$ read on $U$, in particular $W_0$
is an open map. Observe also that by means of the reflections with respect to
the geodesic rays $(0,1)$ and $(0,i)$ issuing from $0$ with asymptotic boundary
$\{1\}$ and $\{i\}$ respectively, the map $W_0$ can be extended to a larger
open set
$\wh U\subset \di$ containing $U$ and the open geodesic rays, and this extended
map is still an open map.

Observe also that $g= 2 \Imag W_0$.

Since $g= h \circ (\Pi \circ X)^{-1}$, we define the ``conjugate function''
$^*g$ setting \linebreak
$^*g :=\,  ^*h \circ (\Pi \circ X)^{-1}$. Since
$g$ has no critical point on $U$, we observe that
$^*g$ is strictly monotonous on any level curve of $g$. From the relation
(\ref{Eq.height}) we get that $^*g = -2 \Real W_0$.

\vskip3mm

\item\label{item.zero} {\em The level curve $L_0:=g^{-1}(\{0 \})$ cannot be a
simple curve with asymptotic boundary the set $\{1, e^{i\theta_0} \}$.
Consequently, the level curve $L_0$ must have one of the following behaviors.
\begin{itemize}
 \item $\pain L_0 = \{ e^{i\theta_0}, i \}$,

\item $\pain L_0 = \{ e^{i\theta_0} \}$ and $L_0$ has an end point on
$[0,1) \cup (0,i)$.
\end{itemize}
}

Let us assume, by absurd, that $\pain L_0 = \{1, e^{i\theta_0}\}$.
Let $p_0 \in L_0$ be a fixed point. We denote by $L_0^+$ and $L_0^-$ the
components of $L_0 \setminus \{ p \}$ with asymptotic boundary $1$ and
$e^{i\theta_0}$ respectively.

Recall that from item (\ref{item.rectangle}) we have fixed a number
$c\in (0,1)$ and from item (\ref{item.rho}) there exists
$\rho_c >0$ such that $\vert n_3 (\xi)\vert < c$ for any
$\xi \in \mathcal{Z}_{\rho_c}$.

We deduce from item (\ref{item.ray}) that, up to extracting a compact part,
we can assume that $L_0^+ \subset \mathcal{Z}_{\rho_c}$ and
$L_0^- \subset \mathcal{Z}_{\rho_c}$.

Since the graph of $L_0^+$ on $M$ has infinite length, the curve
$W_0 (L_0^+)\subset \wt \Omega$ must have infinite length as well, for the
metric $d\wt s ^2$, see (\ref{Eq.metric}).

As $\vert n_3 \vert < c$ on $\mathcal{Z}_{\rho_c}$, we get that the function
$\wt \omega$ is bounded on $\mathcal{Z}_{\rho_c}$. Thus, the curve
$W_0 (L_0^+)\subset \wt \Omega \subset \C$
must have infinite Euclidean length. We deduce that
\begin{equation*}
 \Real W_0 (\xi) \to \pm \infty,\  \  \text{if} \ \xi \to 1,\ \xi\in L_0.
\end{equation*}
Without loss of generality, we can assume that $ \Real W_0 (\xi) \to + \infty$
when $\xi \to 1$, $\xi \in L_0$.

In the same way, and using the fact that $\Real W_0$ is strictly monotonous on
any level curve of $g$, we get that
\begin{equation*}
 \Real W_0 (\xi) \to - \infty,\  \  \text{if} \ \xi \to e^{i\theta_0},\ \xi\in
L_0.
\end{equation*}
Consequently the image of the level curve $L_0$,  $W_0 (L_0)$, is the whole
real axis in $\C$: $W_0 (L_0)= \R \subset \C$.
We arrive to a contradiction as follows.

Note that $g$ can be extended across the geodesic ray $(0,1)$ by means of the
reflection principle. Note also that the critical points of the extended map,
if any, are isolated.

Let $p_1\in (0,i)$ be a fixed point in the geodesic ray which is not a critical
point of $g$. Thus $W_0$ is a local diffeomorphism near $p_1$. Let
$\mathcal{O}_1$ be an open neighborhood of $p_1$ such that $W_0$ is one-to-one
on $\mathcal{O}_1$. Since $W_0 (L_0)= \R \subset \C$, there exists
$p_2 \in L_0$ such that $W_0 (p_1)=W_0 (p_2)$. Now let
$\mathcal{O}_2 \subset U$ be any open neighborhood of $p_2$. As $W_0$ is an
open map, we get that $W_0( \mathcal{O}_1 )\cap  W_0( \mathcal{O}_2 )$ is an
open set containing $W_0(p_2)$, which gives a contradiction with the fact $W_0$
is one-to-one on $U$.

\vskip1.5mm

\item {\em Let $U^+:= \{\xi \in U,\ g(\xi) >0\}$ and let
$\Omega^+:= \{w \in \C, \Imag w >0 \}$. Then
$W_0 (U^+)= \Omega^+\subset \wt \Omega $.
}

Indeed, for any $c >0$ we set  $L_c:= g^{-1}(\{c \}) $. Thus
$W_0(L_c)$ is contained in the horizontal line of $\Omega^+$ at height $2c$:
$W_0(L_c)\subset \{w\in \C, \Imag w =2c  \}  $, since $g=2\Imag W_0$.

We can prove in the same way as in the proof of item (\ref{item.zero}), that
$W_0(L_c)$ is the whole line $\{w\in \C, \Imag w =2c  \}$. We conclude that
$W_0 (U^+)= \Omega^+\subset \wt \Omega $.

\vskip1.5mm

\item {\em Let $p_0\in (0,1)$ be any fixed point of the geodesic ray $(0,1)$,
such that   $u_0 := W_0 (p_0) >0$. We
consider the subset
$\Omega_1 :=\{w\in \Omega^+,\ \Real w >~u_0 \}$. Let
$S\subset M$ be the corresponding part of $M$, that is
$S=\wt X (\Omega_1)$, see the discussion after the item (\ref{item.univalent}).
Then, $S$ has finite total curvature.
}

Since $\wt X : (\wt \Omega, d \wt s ^2 )\rightarrow \h^2 \times \R$ is an
isometric immersion, it is equivalent to prove that  $\Omega_1$ has finite
total curvature with respect to the metric $d \wt s ^2$.

For any  $C > 0$, we consider the square $R(C)\subset \Omega_1$
with
horizontal sides $H_0(C)$ and $H_1(C)$ and vertical sides $V_0 (C)$ and $V_1(C)$
defined by
\begin{itemize}
 \item $H_0(C)$ is the horizontal segment with end points $u_0$ and
$u_0 + C)$,

 \item $H_1(C)$ is the horizontal segment with end points $u_0 +iC$ and
$u_0 + C +iC$,

 \item $V_0(C)$ is the vertical segment with end points $u_0$
and $u_0 +i C$,

 \item $V_1(C)$ is the vertical segment with end points $u_0 +C$ and
$u_0 + C + iC$,
\end{itemize}

The Gauss-Bonnet theorem applied to the square $R(C)$ gives
\begin{equation}\label{Eq.Gauss-Bonnet}
 \int_{R(C)} K\, dA = -\int_{\partial R(C)} k_g\, ds,
\end{equation}
where $dA$ is the area element of $(\wt \Omega, d\wt s^2)$, $K$ is the Gaussian
curvature, $k_g$ is the geodesic curvature along $\partial R(C)$ parametrized
by the  arc length  $s$. Hence it suffices to show  that the right integral of
(\ref{Eq.Gauss-Bonnet}) is bounded if $C \to +\infty$.

\smallskip

We have
\begin{equation*}
 \int_{\partial R(C)} k_g\, ds =  \int_{H_0(C)} k_g\, ds +
 \int_{H_1(C)} k_g\, ds + \int_{V_0(C)} k_g\, ds
 +\int_{V_1(C)} k_g\, ds.
\end{equation*}
Since the geodesic ray $(0,1)$ is a geodesic of $M$, we get that $H_0(C)$ is a
geodesic of $S$ for any  $C>0$. Thus
\begin{equation*}
  \int_{H_0(C)} k_g\, ds =0
\end{equation*}
for any  $C>0$. We are going to prove that the integral
 on $V_0 (C)$ is
bounded when $C\to +\infty$.

We choose the following parametrization of $V_0 (C)$,
\begin{equation*}
 \gamma (t)= u_0 + i\, tC,\ \ t\in [0,1].
\end{equation*}
 Let $w=u+iv$ be the coordinates on $\wt \Omega$. We
deduce from the expression of the metric $d\wt s^2$, see Formula
(\ref{Eq.metric}), and  from
\cite[Formula (42.8)]{Kreyszig}, that the geodesic curvature of the curve
$\gamma$ is given by
\begin{equation}\label{Eq.courbure}
 k_g (\gamma(t))= \pm\frac{\sinh \wt \omega }{2\cosh ^2 \wt \omega}\,
 \frac{\partial \wt \omega}{\partial u } \big(\gamma (t)\big).
\end{equation}
Let $\Theta\subset \C$ be any domain on which the function
$\wt \omega$ is defined and satisfies the equation (\ref{Eq.Laplacian}).
For any $w \in \Theta$ we denote by $d(w, \partial \Theta)$
the Euclidean distance between $w$ and $\partial \Theta$.
It is shown in the proof of \cite[Proposition 2.3]{HNST} that there exists
a positive constant $\delta$ such that for
any $w \in \Theta$ with $d(w, \partial \Theta)> 2$ we have
\begin{equation*}
 \vert \nabla \wt \omega \vert (w) < \delta e^{-d(w,\partial \Theta)},
\end{equation*}
where $\nabla$ means the Euclidean gradient.

Hence, choosing $\Theta= \Omega^+$ we get
\begin{equation}\label{Eq.gradient}
 \vert \nabla \wt \omega \vert (w) < \delta e^{-\Imag w}
\end{equation}
for any $w\in \Omega^+$ such that $\Imag w >2$. For any $C>3$ we have
\begin{align*}
 &\int_{V_0(C)} k_g\, ds = 2\int_0^1 k_g (\gamma (t))
 \ch \wt \omega (\gamma (t)) \, \vert \gamma^\prime (t)\vert\, dt = \\
  &2\int_0^{3/C} k_g (\gamma (t))
 \ch \wt \omega (\gamma (t)) \, \vert \gamma^\prime (t)\vert\, dt +
 2\int_{3/C}^1 k_g (\gamma (t))
 \ch \wt \omega (\gamma (t)) \, \vert \gamma^\prime (t)\vert\, dt.
\end{align*}
Since $V_0(C)$ is a smooth curve, there exists a constant number $\alpha>0$
such that
\begin{equation*}
 2\int_0^{3/C}\vert k_g (\gamma (t))\vert
 \ch \wt \omega (\gamma (t)) \, \vert \gamma^\prime (t)\vert\, dt < \alpha,
\end{equation*}

 for any $C>4$. On the  other hand, from the formulae
(\ref{Eq.courbure}) and
(\ref{Eq.gradient}) we get for any $t >3/C$
\begin{equation*}
2 \vert k_g (\gamma (t))\vert  \ch \wt \omega (\gamma (t))
\leq \vert \nabla \wt \omega\vert (\gamma (t)) <
\delta e^{-t C}.
\end{equation*}
Therefore
\begin{align*}
  2\int_{3/C}^1\vert k_g (\gamma (t))\vert
 \ch \wt \omega (\gamma (t)) \, \vert \gamma^\prime (t)\vert\, dt &\leq
 2\delta C \int_{3/C}^1 e^{-t C} \, dt \\
 &\leq 2\delta C \, \frac{(e^{-3}-e^{-C} )} {C}\\
 &\leq 2\delta (e^{-3}-e^{-C} ) \\
 &\leq 2\delta e^{-3}.
\end{align*}
This proves that the integral $\int_{V_0(C)} k_g\, ds$ is bounded when
$C\to +\infty$.

Choosing again $\Theta= \Omega^+$, we can prove in the same way that
 $\int_{H_1(C)} k_g\, ds \rightarrow 0$ when $C \to +\infty$.

Finally, recall that the minimal surface $M$ can be extended across the
geodesic ray $(0,1)$ by means of the
reflection principle. Therefore the function $\wt \omega$ can be extended to
the domain
$\Theta := \{z\in \C,\ \Re (z) > u_0\}$. Consequently we can prove that
 $\int_{V_1(C)} k_g\, ds \rightarrow 0$ when $C \to +\infty$.

 We conclude that  $S$ has finite total curvature.

\end{enumerate}

\section{Some examples of infinite total curvature minimal surfaces
in $\hip^2\times\R$ and their asymptotic boundary}

Next we exhibit complete and non-complete minimal surfaces $M$ generated
by vertical graphs, pointing out some geometric properties. For this purpose we
choose the Poincar\'e disc model of $\h^2$.

\begin{example}\label{Ex.geodesic}
{\em $M$ is non-complete, properly embedded and its asymptotic
boundary is the union
of a discrete set in $\partial_\infty \hi2 \times \R$,
with

\begin{itemize}
\item either the whole
$ (\hi2 \cup \pain \h^2)  \times \{-\infty, + \infty\}$,

\item  or
a finite subset of $\pain \hi2  \times \{-\infty, + \infty\}$ and
$ \bigcup_{i=1}^n \gamma_i \times \{-\infty, + \infty\}$, where
$\gamma_i\subset \hi2$ is a complete geodesic, $i=1,\dots, n$.
\end{itemize}
}

\smallskip

To obtain such a surface we consider, for $\theta \in (0,\pi/2)$ the geodesic
triangle $T$ with vertices $0$, $1$ and $e^{i\theta}$. Let $c >0$ be a positive
real number. Let $\gamma \subset \hd$ be the geodesic   with asymptotic
boundary the points 1 and  $e^{i\theta}$.

Let $f : \gamma \rightarrow \R$ be a continuous and one to one
function, such that $f(\xi)\to 0$ if $\xi \to e^{i\theta}$ and
$f(\xi)\to c$ if $\xi \to 1$.

We consider the Dirichlet problem for the minimal surface
equation on ${\rm int} (T)$ taking the boundary data

\begin{align*}
- & \  c \ \ \text{on the geodesic ray } (0,1), \\
- & \ 0 \ \ \text{on the geodesic ray } (0,e^{i\theta}), \\
- & \  f\  \  \text{on} \ \gamma .
\end{align*}

We deduce from \cite[Theorem 4.1]{SE-T2} that there exists a solution
$u$ to this problem. Thus the graph $S$ of $u$ is a minimal surface
whose  boundary contains the geodesic rays $(0,1)\times \{c\}$ and
$(0,e^{i\theta})\times \{0\}$,
 the vertical segment
$\{(0,t) \in \hi2 \times \R,\ 0\leq t \leq c\}$ and
the graph of $f$ on $\gamma$.
Now we perform the reflection of $S$ with
respect to the geodesic rays $(0,1)$, $(0,e^{i\theta})$ and the new geodesic
rays appearing in this process.

In this way we get a non complete and properly embedded minimal surface $M$
invariant by a discrete group of screw-motions.
The finite
asymptotic boundary is a discrete set of $\h^2 \times \R$

To describe the non finite asymptotic boundary of $M$ we consider two cases.
\begin{itemize}
 \item If the angle $\theta/\pi$ is irrational then the non finite
asymptotic boundary is the whole
$ (\hi2 \cup \pain \h^2 ) \times \{-\infty, + \infty\}$.

\item   If the angle $\theta/\pi$ is rational then the non finite
asymptotic
boundary is composed of a finite subset
$\{\pm \xi_1,\dots, \pm\xi_n\}\times \{-\infty, + \infty\}$ of
$\pain \hi2  \times \{-\infty, + \infty\}$ and
$ \bigcup_{i=1}^n \gamma_i \times \{-\infty, + \infty\}$, where
$\gamma_i\subset \hi2$ is the complete geodesic with asymptotic boundary
$\{-\xi_i, \xi_i\}$, $i=1,\dots, n$.
\end{itemize}
\end{example}

\begin{example}\label{Ex.helix}
 {\em  $M$ is complete, properly embedded and its  finite
asymptotic
 boundary consists in the
 union of two helix type curves.
  The rest of the asymptotic boundary
 consists
\begin{itemize}
 \item either in the whole
 $ (\hi2 \cup \pain \h^2)  \times \{-\infty, + \infty\}$,

\item or
a finite subset of $\pain \h^2  \times \{-\infty, + \infty\}$ and
$ \bigcup_{i=1}^n \gamma_i \times \{-\infty, + \infty\}$, where
$\gamma_i\subset \hi2$ is a complete geodesic, $i=1,\dots, n$.
\end{itemize}
}

\medskip

Let $\theta \in (0,\pi/2)$ be a fixed number. We denote by
$\Gamma_\theta \subset \pain \h^2$ the closed arc of $\pain \h^2$
bounded by $1$ and $e^{i\theta}$ which does not contain $i$.

Let $D_\theta \subset \h^2$ be the domain bounded by the geodesic rays
$(0,1)$ and $(0,e^{i\theta})$ and whose  asymptotic boundary is
$\Gamma_\theta$.

Let $f : \Gamma_\theta \rightarrow \R$ be a continuous and one to one
function, such that $f(e^{i\theta})=0$ and  and
$f(1)=c$, where $c>0$ is a positive number.

We consider the Dirichlet problem for the minimal surface
equation on $D_\theta$ taking the boundary data

\begin{align*}
- &\  c \  \ \text{on the geodesic ray } (0,1), \\
- &\  0 \ \ \text{on the geodesic ray } (0,e^{i\theta}), \\
- &\   f\   \ \text{on} \ \Gamma_\theta .
\end{align*}

We deduce from \cite[Theorem 4.1]{SE-T2} that there exists a solution
$u$ to this problem. Thus the graph $S$ of $u$ is a minimal surface
whose  boundary contains the geodesic rays $(0,1)$ and $(0,e^{i\theta})$
 and the vertical segment
$\{(0,t) \in \hi2 \times \R,\ 0\leq t \leq c\}$.

 Now we perform the reflection of $S$ with respect to the vertical geodesic
$\{ 0 \} \times \R$ and with
respect to the geodesic rays $(0,1)$, $(0,e^{i\theta})$ and the new geodesic
rays appearing in this process.

In this way we get a complete and  properly embedded minimal surface $M$
invariant by a
discrete group of screw-motions.
The finite
asymptotic boundary is composed of to ``helix type'' curves.
  We deduce from Theorem \ref{Main theorem} that $M$ has infinite
total curvature.

To describe the non finite asymptotic boundary of $M$ we consider two cases.
\begin{itemize}
 \item If the angle $\theta /\pi$ is irrational then the non finite
asymptotic
boundary is the whole
$ (\hi2 \cup \pain \h^2 ) \times \{-\infty, + \infty\}$.

\item If the angle $\theta/\pi$ is rational then the non finite
asymptotic
boundary is composed of
  a finite subset
$\{\pm \xi_1,\dots, \pm\xi_n\}\times \{-\infty, + \infty\}$ of
$\pain \hi2  \times \{-\infty, + \infty\}$ and
$ \bigcup_{i=1}^n \gamma_i \times \{-\infty, + \infty\}$, where
$\gamma_i\subset \hi2$ is the complete geodesic with asymptotic boundary
$\{-\xi_i, \xi_i\}$, $i=1,\dots, n$.
\end{itemize}
\end{example}

\begin{example}

{\em
$M$  is non properly immersed and
 its asymptotic boundary is an
 annulus $\partial_\infty \hi2 \times [-a,a]$, where $a >0$.
}

\medskip

In order to construct  such an example, we proceed as in Example \ref{Ex.helix}
above, setting $c=0$, $f(1)=a$, $f(e^{i\theta})=0$ and  $\theta /\pi$ is
irrational. Observe that this
surface is complete far away from the origin.

\end{example}

\begin{example}
{\em The asymptotic boundary of $M$ is either
$\partial_\infty (\hi2 \times \R)\setminus (D\times \{-\infty, +\infty\})$,
where $D$ is an
open geodesic disc of $\hi2$ or the whole asymptotic boundary of $\hi2 \times
\R$.}

\medskip

We proceed as in Example \ref{Ex.geodesic} above, now $c\geq 0$ is a
nonnegative constant and $f \equiv  +\infty$ on the geodesic $ \gamma$.

It can be shown, using \cite[Theorem 4.1]{SE-T2} that this Dirichlet problem
has a solution.

We choose $\theta$ such that $\theta/\pi$ is irrational.

After performing all reflections we get a minimal surface $M$. To describe
the surface $M$ we consider two cases.
\begin{itemize}
 \item $c=0$. In this case $M$ is complete far away from the origin and is
 non properly immersed. Its asymptotic boundary is
$\partial_\infty (\hi2 \times \R)\setminus (D\times \{-\infty, +\infty\})$,
  where $D \subset \hi2$ is the open geodesic disc centered at 0
and such that $D\cap \gamma = \emptyset$ and
$\ov D \cap \gamma \not= \emptyset$.

\item  $c >0$. In this case $M$ is complete, properly immersed,
and its
asymptotic boundary is the whole asymptotic boundary of $\hi2 \times \R$.
\end{itemize}
\end{example}

\begin{example}
{\em $M$ is complete and dense in $\hi2 \times \R$.

Therefore its
 asymptotic boundary is the whole asymptotic boundary of $\hi2 \times \R$.}

\medskip

 We proceed as in Example \ref{Ex.geodesic} above with the following
modifications.
\begin{itemize}
 \item $f \equiv  +\infty$ on the geodesic $ \gamma$.

 \item  On the geodesic ray $(0,e^{i\theta})$ we consider the
constant boundary data 0.

 \item On the geodesic ray $(0,1)$ we consider the boundary data $g$
 given by
\begin{equation*}
 g=\begin{cases}
    \pi \ \text{on}\ (0,1/3) \\
    1 \ \text{on}\ (1/3, 2/3) \\
    0 \ \text{on}\ (2/3,1)
   \end{cases}
\end{equation*}
\end{itemize}
  We choose $\theta$ such that  $\theta/\pi$ is irrational.

It can be shown, using \cite[Theorem 4.1]{SE-T2} that this Dirichlet problem
has a solution.

The  complete minimal surface $M$ obtained by doing all
reflections is dense.
\end{example}

\vskip1.5mm

 All the minimal examples above have  infinite total curvature.


\begin{thebibliography}{999}


 \bibitem{B-SE} \textsc{P. B\'erard and R. Sa Earp},
 {\rm Minimal hypersurfaces
in $\hi n \times \r$, total curvature and index}.
{\em Bollettino dell'Unione Matematica Italiana.} First online: 15
January 2016. Doi: 10.1007/s40574-015-0050-0.


\bibitem{Cheng-Tysk} \textsc{S.-Y. Cheng and J. Tysk,
{\rm Schr\"odinger operators and index bounds for minimal submanifolds},
{\em Rocky Mountain J. Math.} 24 (1994), no. 3, 977--996.}

\bibitem{C-M} \textsc{T.H. Colding, W.P. Minicozzi}
{\em A course in minimal surfaces},
Graduate Studies in Mathematics, \textbf{121}. American Mathematical Society,
Providence, RI, 2011.


\bibitem{C-R} \textsc{P. Collin and H. Rosenberg}, {\rm Construction of
harmonic diffeomorfisms and minimal graphs}, {\em Annals of Mathematics}
\textbf{172} (3) (2010), 1879--1906.

\bibitem{Coskunuzer}\textsc{B. Coskunuzer},
{\rm Minimal surfaces with arbitrary topology in $\hi2 \times \R$},
arXiv:1404.0214v2, 2014.


\bibitem{F-M} \textsc{I. Fern\'andez and P. Mira}.
{\rm Harmonic Maps and Constant Mean Curvature Surfaces in $\hip^2 \times \R$},
{\em American Journal of Mathematics} \textbf{129} (4) (
2007),  1145--1181.

 \bibitem{F-C} \textsc{D. Fisher-Colbrie},
 {\rm On complete minimal
 surfaces with finite Morse index in three manifolds},
 {\em Inventiones Mathematicae} \textbf {82} (1) (1985), 121--132.

\bibitem{Frensel} \textsc{K.R. Frensel},
{\rm Stable complete surfaces with
constant mean
curvature}, {\em
Bulletin of the Brazilian Mathematical Society} \textbf{ 27} (2) (1996),
129--144.


\bibitem{G-R}   \textsc{J A. G\'alvez, H. Rosenberg},
{\rm  Minimal surfaces and harmonic
diffeomorphisms from the complex plane onto certain Hadamard surfaces},
{\em American Journal of Mathematics} \textbf{132} (5) (2010), 1249--1273.


\bibitem{Grig-Yau}  \textsc{A. Grigor'yan and S.-T. Yau},
{\rm Isoperimetric properties of higher eigenvalues of elliptic operators},
{\em Amer. J. Math.} 125 (2003), no. 4, 893--940.

\bibitem{H-M} \textsc{ L.Hauswirth and A. Menezes},
{\rm On doubly periodic minimal surfaces in $\hip^2 \times \R$ with finite total
curvature in the quotient space}, {\em Annali Di Matematica Pura
Ed Applicata}. First online: 23 August 2015. Doi: 10.1007/s10231-015-0524-9.

\bibitem{HNST} \textsc{L. Hauswirth, B. Nelli, R. Sa Earp and E. Toubiana},
{\rm A Schoen theorem for minimal surfaces  in ${\mathbb H}^2\times {\mathbb
R}$}, {\em Advances in Mathematics} \textbf{274} (2015) 199--240.

\bibitem{HR} \textsc{L. Hauswirth and H. Rosenberg},
{\rm Minimal surfaces of
finite total  curvature  in $\h^2\times\r,$ } {\em  Matematica Contemporanea}
\textbf {31} (2006), 65--80.

\bibitem{HST} \textsc{L. Hauswirth, R. Sa Earp and E. Toubiana},
{\rm Associate and
  conjugate minimal immersions in $M\times\r,$ } {\em Tohoku Mathematical
Journal} \textbf {60} (2) (2008),
 267--286.



\bibitem{Hu} \textsc{ A. Huber},
{\rm On subharmonic functions and differential geometry in the large,}
{\em Comment. Math. Helv.}, \textbf{32} (1957),  13-–72.


\bibitem{K-M}\textsc{B. Kloeckner, R. Mazzeo},
{\rm On the asymptotic behavior of minimal surfaces in $\hi2 \times \R$},
arXiv:1506.02838v1, 2015.

\bibitem{Kreyszig} \textsc{E. Kreyszig},
 {\em Introduction to Differential
 Geometry and Riemannian Geometry}, Translated from the German, Mathematical
 Expositions {\textbf 16}, University of Toronto Press, Toronto,  1968.

\bibitem{Ma-R}\textsc{F. Martin, M. M. Rodriguez :}
{\em Minimal planar domains
in $\hd \times \R$}, Transactions of the AMS 365 (2013), 6167--6183.


\bibitem{Mo-R} \textsc{F. Morabito and M. Rodriguez},
{\rm Saddle towers and
minimal
$k$-noids in $\h^2\times\r,$}     {\em Journal of the Institute of Mathematics
of
Jussieu} \textbf {11} (2) (2012),  1--17.

\bibitem{N-R0} \textsc{B. Nelli and H. Rosenberg},
{\rm Minimal  Surfaces in ${\h}^2\times {\re},$} {\em Bulletin of the Brazilian
Mathematical Society} {\bf 33}
(2002),  263--292.

\bibitem{N-R} \textsc{B. Nelli and H. Rosenberg},
  Errata Minimal Surfaces in
$\h^2\times\re,$
{\em Bulletin of the Brazilian Mathematical Society, New Series} \textbf{38}
(4)
(2007),1--4.






\bibitem{NST}  \textsc{B. Nelli, R. Sa Earp and E. Toubiana},
 {\rm Maximum Principle and Symmetry for Minimal Hypersurfaces
in ${\mathbb H}^n\times {\mathbb R}$}, {\em Annali della Scuola Normale
Superiore di Pisa, Classe di Scienze}.
Vol. XIV (2015) 1--14.



\bibitem{Osserman}\textsc{R. Osserman}, {\em  A survey on minimal surfaces,}
 Dover Publications, New York, 1986.


\bibitem{P} \textsc{J. Plehnert}. {\rm  Constant mean curvature $k$ noids in
homogeneous manifolds},
{\em Illinois Journal of Mathematics} {\bf 58} (1) (2014),  233--249.


\bibitem{RST} \textsc{H. Rosenberg, R. Souam and E. Toubiana},
{\rm General curvature estimates for stable $H$-surfaces in 3-manifolds and
applications},
{\em Journal of Differential Geometry} {\bf 84} (2010), 623--648.

\bibitem{S-E} \textsc{R. Sa Earp},
{\rm Parabolic and Hyperbolic Screw motion in $\h^2\times\r,$} {\em Journal of
the
Australian  Mathematical  Society} \textbf{85} (2008), 113--143.




\bibitem{SE-T1} \textsc{R. Sa Earp and E. Toubiana},
{\rm Screw motion surfaces in $\hd\times \R$ and $\sd\times \R$}, {\em
Illinois Journal of Mathematics} \textbf{49} (2005), 1323--1362.

\bibitem{SE-T2} \textsc{R. Sa Earp and E. Toubiana},
{\rm  An asymptotic theorem for
minimal surfaces and existence results
for minimal graphs in $\hip^2 \times \R$}, {\em
Mathematische Annalen},
\textbf{342} (2) (2008), 309--331.

\bibitem{Livre}  \textsc{R. Sa Earp, E. Toubiana:} {\em Introduction \`a la
g\'eom\'etrie hyperbolique et aux surfaces de Riemann,} Cassini (2009).

\bibitem{SE-T3}  \textsc{R. Sa Earp, E. Toubiana :}
{\em A minimal stable vertical planar minimal end in $\hd\times \R$ has finite
total curvature}, Journal of the London Mathematical Society
\textbf{92} (3), 712--723 (2015).

\bibitem{SE-T4}  \textsc{R. Sa Earp and E. Toubiana},
{\rm Minimal graphs in
$\hi n \times \r$ and $\r^{n+1}$}, {\em Annales de l'Institut Fourier}
\textbf{60} (7) (2010), 2373--2402.

\bibitem{eBook} \textsc{R. Sa Earp and E. Toubiana}.  {\em Topologie, courbure
et
structure conforme sur les surfaces}.  RG, 2015 (eBook-open access).
Doi: 10.13140/RG.2.1.3623.1769.


\end{thebibliography}
\end{document}